\documentclass[preprint,12pt]{elsarticle}
\usepackage{amsthm,amsfonts,amssymb,amscd,amsmath,enumerate,verbatim,calc,graphicx,geometry}
\usepackage[all]{xy}
\newtheorem{theorem}{Theorem}[section]
\newtheorem{lemma}[theorem]{Lemma}
\newtheorem{proposition}[theorem]{Proposition}
\newtheorem{corollary}[theorem]{Corollary}
\theoremstyle{definition}
\theoremstyle{definitions}
\newtheorem{definition}[theorem]{Definition}

\newtheorem{remark}[theorem]{Remark}
\newtheorem{example}[theorem]{Example}
\theoremstyle{notations}

\theoremstyle{remarks}

\journal{ }

\begin{document}

\begin{frontmatter}



\title{On Generalized Covering Groups of Topological Groups}


\author[]{Hamid~Torabi\corref{cor1}}
\ead{h.torabi@ferdowsi.um.ac.ir}
\author[]{Mehdi~Abdullahi~Rashid}
\ead{mbinev@mail.um.ac.ir}
\author[]{Majid~Kowkabi}
\ead{m.kowkabi@stu.um.ac.ir}

\address{Department of Pure Mathematics, Ferdowsi University of Mashhad,\\
P.O.Box 1159-91775, Mashhad, Iran.}
\cortext[cor1]{Corresponding author}

\begin{abstract}
   It is well-known that a homomorphism $ p : \widetilde{G} \rightarrow G $ between topological groups is a covering homomorphism if and only if  $ p $ is an open epimorphism with discrete kernel. In this paper we generalize this fact, in precisely, we show that for a connected locally path connected topological group $ G $, a continuous map $ p : \widetilde{G} \rightarrow G $  is a generalized covering if and only if $ \widetilde{G}  $ is a topological group and $ p $ is an open epimorphism with  prodiscrete (i.e, product of discrete groups) kernel. To do this we first show that if $ G $ is a topological group and $ H $ is any generalized covering subgroup of $ \pi_1(G,e) $, then $ H $ is as intersection of all covering subgroups, which contain $ H $. Finally, we show that every generalized covering of a connected locally path connected topological group is a fibration.
\end{abstract}


\end{frontmatter}

%\\\\\\\\\\\\\\\\\\\\\\\\\\\\\\\\\\\\\\\\\\\\\\\\\\\\\\\\\\\\\\\\\\\\\\\\\\\\\\\\\\\\\\\\\\\\\\\\\\\\\\\\\\\\\\\\\\\\\\\\\\\\\\\\\\\\\\\\\
%=========================================================================================================================================
%/////////////////////////////////////////////////////////////////////////////////////////////////////////////////////////////////////////
\section{Introduction}

  Chevalley \cite{Chev} introduced a covering group theory for connected, locally path connected, and semi locally simply connected topological groups. Rotman \cite[Theorem 10.42]{Rotm} proved that for every covering space $(\widetilde{X}, p)$ of a connected, locally path connected, and semi locally simply connected topological group $G$, $\widetilde{X}$ is a topological group and $p$ is a homomorphism. Recently, Torabi \cite{Torabi} developed this theory for connected locally path connected topological groups and gave a classification for covering groups of them. He showed that the natural structure, which has been called the path space, and its relative endpoint projection map for a connected locally path connected topological group are a topological group and a group homomorphism, respectively. Recall that for an arbitrary subgroup $ H $ of the fundamental group $ \pi_1(X,x_0) $, the path space is the set of all paths starting at $ x_0 $, which is denoted by $ P(X,x_0) $,  with an equivalence relation $\sim_H$ as follows:  
  
 ${\alpha }_1\sim_H {\alpha }_2$ if and only if ${\alpha }_1(1)={\alpha }_2(1)$ and $[{\alpha }_1*{{\alpha }_2}^{-1}]\in H$.
The equivalence class of $\alpha$ is denoted by ${\left\langle \alpha \right\rangle }_H$.
One can consider the quotient space $\widetilde{X}_H=P(X,x_0)/\sim_H$ and the endpoint projection map $p_H: (\widetilde{X}_H,e_H)\rightarrow (X,x_0)$ defined by ${\left\langle \alpha \right\rangle }_H \mapsto \alpha (1)$, where $e_H$ is the class of the constant path at $x_0$.

If $\alpha \in P(X,x_0)$ and $U$ is an open neighbourhood of $\alpha (1)$, then a continuation of $\alpha$ in $U$ is a path $\beta=\alpha *\gamma $, where $\gamma $ is a path in $U$ with $\gamma (0)=\alpha (1)$. Put $ N({\left\langle \alpha \right\rangle }_H,U) =\{{\left\langle \beta \right\rangle }_H\in {\widetilde{X}}_H \ | \ \mathrm{\beta\ is\ a\ continuation\ of\ \alpha\ in\ U}\}$. It is well known that the subsets $ N({\left\langle \alpha \right\rangle }_H, U) $ form a basis for a topology on ${\widetilde{X}}_H$ for which the function $p_H:{(\widetilde{X}}_H,e_H)\rightarrow (X,x_0)$ is continuous (see \cite[Page 82]{Span}). Brodskiy et al. \cite{BroU} called this topology on ${\widetilde{X}}_H$ the whisker topology.
  
  Moreover, they were interested in studying the spaces whose local properties can extend to the entire space and introduced the notions of $ SLT $ and strong $ SLT $ spaces.
  
  \begin{definition} \cite[Definition 4.18]{BroU}
  A topological space $X$ is called \textit{strong small loop transfer space (strong $ SLT $ space for short) at $x_{0}$} if for each point $x \in X$  and for every neighbourhood $U$ of $x_{0}$, there is a neighbourhood $V$ of $x$ such that for every path $ \alpha: I \rightarrow X $ from $ x_0 $ to $ x $ and every loop $\beta$ in $V$ based at $x$ there is a loop $\gamma$ in $U$ based at $x_{0}$  which is homotopic to $ \alpha\ast\beta\ast\alpha^{-1}$ relative $\dot{I}$. The space $X$ is called strong $ SLT $ space, if $X$ is strong $ SLT $ space at $x$ for every $x \in {X}$.
 \end{definition}
 
 On the other hand, Torabi   showed  that every topological group is a strong $ SLT $ at the identity element \cite[Theorem 2.4]{Torabi}. 
  In Section \ref{sec3}, we introduce the generalized covering group of a topological group and give some examples to clarify the difference between covering and generalized covering groups (Examples \ref{ex3.12} and \ref{ex313}). Moreover, we show that for a connected locally path connected strong $ SLT $ at $ x_0 $ space, every generalized covering subgroup $ H $ of the fundamental group can be written as the intersection of all covering subgroups, which contains $ H $ and vice versa. 
  
   In Section \ref{se4}, we attempt to provide a method for classifying generalized covering groups of a topological group by studying the topology of the kernel of the relative generalized covering homomorphism. Of cores, we extend the well-known result about covering groups (Remark \ref{reNew}) for generalized covering groups and show that if $ G $ is a connected locally path connected topological group, $ (\widetilde{G},p) $ is a generalized covering group of $ G $ if and only if $ p $ is an open epimorphism with prodiscrete (i.e, product of discrete groups) kernel (Corollary \ref{co310}). In this regard, we first show that  the central fibre of a generalized covering map of an arbitrary topological space is totally path disconnected ( Proposition \ref{pr37}). Counterexample \ref{ex315} show that it is not a sufficient condition for generalized covering subgroups, even in the case of topological groups. After that we present our desirable definition of prodiscrete subgroups (Definition \ref{de43}) and show that for a connected locally path connected topological group $ G $ and a generalized covering subgroup $ H \leq \pi_1(G,e) $, the kernel of the endpoint projection homomorphism is a prodiscrete subgroup ( Theorem \ref{th43}). 
   
 Berestovskii et al. in \cite{Beres} provided a new definition by extending the concept of a cover of a topological group $ G $ such as the pair $ (\widetilde{G},p) $, where $ \widetilde{G} $ is a topological group and the homomorphism $ p : \widetilde{G} \rightarrow G$ is an open epimorphism with prodiscrete kernel. Note that the meaning of prodiscrete kernel in the sense of them was as the inverse limit of discrete groups. In this paper we use Definition \ref{de43} for a prodiscrete concept and show that if $ H $ is a prodiscrete normal subgroup of topological group $ G $, then the pair $ (G,\varphi_H) $ is a generalized covering group of $ \frac{G}{H} $ (Theorem \ref{th44}). Using this theorem, we conclude the main result in Corollary \ref{co310}. Finally, we show that in the case of topological groups the concepts of rigid covering fibrations (which was firstly introduced by Biss \cite{Biss}) and generalized covering groups are coincide and conclude that every generalized covering group of a topological group is also a fibration.

%\\\\\\\\\\\\\\\\\\\\\\\\\\\\\\\\\\\\\\\\\\\\\\\\\\\\\\\\\\\\\\\\\\\\\\\\\\\\\\\\\\\\\\\\\\\\\\\\\\\\\\\\\\\\\\\\\\\\\\\\\\\\\\\\\\\\\\\\\\\\\\\\\\\\\\\\\\\\\\\\\\\\\\\\\\\\\\\\\\\\\\\\\\\\\\\\\\\
%\\\\\\\\\\\\\\\\\\\\\\\\\\\\\\\\\\\\\\\\\\\\\\\\\\\\\\\\\\\\\\\\\\\\\\\\\\\\\\\\\\\\\\\\\\\\\\\\\\\\\\\\\\\\\\\\\\\\\\\\\\\\\\\\\\\\\\\\\\\\\\\\\\\\\\\\\\\\\\\\\\\\\\\\\\\\\\\\\\\\\\\\\\\\\\\\\\\\\\\\\\

\section{Generalized Coverings of Topological Groups} \label{sec3}

It is well-known that a continuous map $p:(\widetilde{X},{\tilde{x}}_0)\rightarrow (X,x_0)$ has the \textit{unique lifting property},  if for every connected, locally path connected space$\ (Y,y_0)$ and every continuous map $f:(Y,y_0)\rightarrow (X,x_0)$ with
$f_*{\pi }_1(Y,y_0)\subseteq p_*{\pi }_1(\widetilde{X},{\tilde{x}}_0)$ for ${\tilde{x}}_0\in p^{-1}(x_0)$, there exists a unique continuous map $\tilde{f}:(Y,y_0)\rightarrow (\widetilde{X},{\tilde{x}}_0)$ with $p\circ\tilde{f}=f$. If $\widetilde{X}$ is a connected, locally path connected space and $p:\widetilde{X}\rightarrow X$ has unique lifting property, then $p$ and $\widetilde{X}$ are called a generalized covering map and a generalized covering space for $X$, respectively. 
%Moreover, a continuous map $p:(\widetilde{X},{\tilde{x}}_0)\rightarrow (X,x_0)$ has the \textit{unique path lifting property} (\textit{UPL} for abbreviation) if it has UL property with respect to every path $ \alpha : I \rightarrow X $.
\begin{definition}
Let $ G $ be a topological group. By a generalized covering group of $ G $, a pair $ (\widetilde{G},p) $ is composed of a topological group $ \widetilde{G} $ and  a homomorphism $ p:(\widetilde{G},\tilde{e})  \rightarrow (G,e) $ such that $ (\widetilde{G},p) $ is a generalized covering space of $ G $. 
%The category of all generalized covering groups of a topological group $ G $ is denoted by $ GCTG(G) $.
\end{definition}

It is easy to check that for an arbitrary pointed topological space $ (X,x_0) $ if $ p:(\widetilde{X},\tilde{x}_0)  \rightarrow (X,x_0) $ is a generalized covering space, then the induced map $ p_* : \pi_1 (\widetilde{X},\tilde{x}_0) \rightarrow \pi_1(X,x_0) $ is  one to one. Therefore, the image  $ H=p_*\pi_1 (\widetilde{X},\tilde{x}_0) $ is a subgroup of  $ \pi_1(X,x_0) $, which is called \textit{generalized covering subgroup}. 

\begin{remark} \label{reNew}
There is a well-known result about covering groups of topological groups which we will extend it for generalized covering groups. A homomorphism $ p : \widetilde{G} \rightarrow G $ between topological groups is a covering homomorphism if and only if  $ p $ is an open epimorphism with discrete kernel.
\end{remark}

As mentioned in \cite{paper1} and \cite{FiZa}, the endpoint projection map $ p_H : \widetilde{X}_H \rightarrow X $ is surjective and open if $ G $ is path connected and locally path connected, respectively. On the other hand, for an arbitrary pointed topological space $ (X,x_0) $, Brazas \cite[Lemma 5.10]{BrazG} showed the relationship between the image of a generalized covering map and the space $ \widetilde{X}_H $ as follows:

\begin{lemma} \label{le33}
Suppose that $\hat{p}:(\hat{X},\hat{x})\rightarrow (X,x_0)$ has the unique lifting property with ${\hat{p}}_*({\pi }_1(\hat{X},\hat{x}))=H$. Then there is a homeomorphism $h:(\hat{X},\hat{x})\rightarrow ({\widetilde{X}}_H,e_H)$ such that $p_H\circ h=\hat{p}$.
\end{lemma}

\begin{remark} \label{re}
It has been indicated in \cite[proof of Theorem 3.6]{Torabi} that for every subgroup $ H \leq \pi_1(G,e_G) $ of the fundamental group of a topological group $ G $, one can construct a multiplication on $ \widetilde{G}_H $, which makes $ \widetilde{G}_H $ a topological group and the continuous map $ p_H: \widetilde{G}_H \rightarrow G $ a homomorphism. 
\end{remark}

The following proposition can be obtained by using the above remark and lemma.

\begin{proposition} \label{pr34}
If  $ p:(\widetilde{G},\tilde{e})  \rightarrow (G,e) $ is a continuous map with unique lifting property, $ G$ and $ \widetilde{G} $ are tow topological groups, and $ \widetilde{G} $ is connected locally path connected, then $ (\widetilde{G},p) $ is a generalized covering group.
\end{proposition}

\begin{proof}
If $ p:(\widetilde{G},\tilde{e})  \rightarrow (G,e) $ has unique lifting property, then by the above lemma, there is a homeomorphism $ h: (\widetilde{G},\tilde{e}) \rightarrow (\widetilde{G}_H,e_{\widetilde{G}_H }) $ such that $ p_H \circ h = p $.  As it was discussed in \cite[proof of Theorem 3.6]{Torabi}, $ \widetilde{G}_H $ and so $ \widetilde{G} $ are topological groups. Moreover, Since the map $ p_H $ is a group homomorphism, then for $ g_1, g_2 \in  \widetilde{G}$, 
$ p(g_1g_2)=p_H \circ h (g_1g_2) =  p_H (h (g_1g_2)) =  p_H(h(g_1)h(g_2)) = p_H(h(g_1))p_H(h(g_2)) = p_H \circ h (g_1)p_H \circ h (g_2)= p(g_1)p(g_2) $. Therefore, $ p $ is a homomorphism.
\end{proof}

\begin{corollary} \label{co35}
Let $ G $ be a topological group, and let $ p:(\widetilde{G},\tilde{e})  \rightarrow (G,e) $. Then  $ (\widetilde{G},p) $ is a generalized covering group if and only if it is a generalized covering space, or equivalently, $ p $ is a surjection with unique lifting property.
\end{corollary}

\begin{proof}
It is the immediate result of the definition.
\end{proof}

%As it was referred to in \cite{Torabi} for a connected and locally path connected topological group $ G $ and an arbitrary subgroup $ H \leq \pi_1(G,e) $, the space $ \widetilde{G}_H $ and the endpoint projection map $ p_H: \widetilde{G}_H \rightarrow G $, with the introduced action are a topological group and a homomorphism, respectively. The following corollary is concluded from this fact and Lemma 2.12 from \cite{paper1}.

%\begin{corollary}
%A subgroup $ H \leq \pi_1(G,e) $ is a generalized covering subgroup if and only if $ {(p_H)}_*\pi_1(\widetilde{G}_H,e_{\widetilde{G}_H}) = H $.
%\end{corollary}

Clearly, in the case of semilocally simply connected spaces the category of covering and generalized covering spaces are equivalent. Out of semilocally simply connected spaces there may be a generalized covering, which is not a covering. It seems interesting to find some examples in the case of topological groups. 

\begin{example} \label{ex3.12}
Let $ Y $ be the Tychonoff product of the infinite number of $ \mathbb{S}^1 $'s. Clearly, $ Y $ is a topological group since the product of topological groups is also a topological group. Moreover, every open neighbourhood of any $y \in Y $ contains many $ \mathbb{S}^1 $'s except finite number and so $ Y$ is  not a semilocally simply connected space at  $ y \in Y $. Since $ Y $ is locally path connected, hence $ Y $  does not have the classical universal covering space.  Then the universal path space $ p: {\prod}_{i\in I} \mathbb{R} \rightarrow Y $ is not a covering group of $ Y $. Although, it is a generalized covering group. Since $ p_i: \mathbb{R} \rightarrow \mathbb{S}^1 $ is a covering map for every $ i \in I $, then it is also a generalized covering map. Therefore, from \cite[Lemma 2.31]{BrazG} the product of $ p_i $'s, $ p $, is also a generalized covering map.

%Let $ \alpha, \beta: (I,0) \rightarrow  (\prod_{i \in I} \mathbb{R},\mathbr{0}) $ be two paths with $ p \circ \alpha = p \circ \beta $. Take $ \alpha_i , \beta_i: I \rightarrow \mathbb{R} $ be the $ i $'th axis of $ \alpha , \beta $, respectively. Clearly, $ p_i: \mathbb{R} \rightarrow \mathbb{S}^1 $ is a covering map, hence it has unique path lifting property. Since $ p_i \circ \alpha_i = p_i \circ \beta_i $ and $ \alpha_i(0)=\beta_i(0) $, then $ \alpha_i = \beta_i $. It implies that $ \alpha = \beta $.

%On the other hand, $ \prod_{i \in I} \mathbb{R}  $ is a simply connected space and so $ \pi_1(\prod_{i \in I} \mathbb{R},\mathbr{0}) = p_* \pi_1(\prod _{i \in I} \mathbb{R},\mathbr{0}) = 1 \leq \pi_1(Y,y_0) $ which implies that $ \pi_1^{gc}(Y,y_0) $ is the trivial subgroup. Therefore, $  \pi_1^{sp}(Y,y_0) = \pi_1^{gc}(Y,y_0) =1 $, which shows that $ Y $ is not coverable.
\end{example}

%It is well-known that if $ K \leq \pi_1(X,x) $ is a covering subgroup and $ K \leq H $, then $ H $ is also a covering subgroup. On the other hand, it is clear that for an arbitrary space $ X $, $\pi_1^{sp}(X,x)  $ is the intersection of all of covering subgroups of $ \pi_1(X,x) $.  It seems  interesting to investigate some proper subgroup $ H \leq  \pi_1(X,x) $ such that $ \pi_1^{sp}(X,x) < H $ and $ H $ is not a covering subgroup. 

In the above example, $ p_* \pi_1(\prod_{i \in I} \mathbb{R}, \bold{0}) $ is a trivial subgroup of $ \pi_1(Y,\mathbf{0}) $. It seems interesting to introduce a nontrivial generalized covering subgroup of $ \pi_1(Y,\mathbf{0}) $, which is not a covering subgroup.

\begin{example} \label{ex313}
From the above example, let $ q: \prod_{i \in I}{{\mathbb{R}} \times \mathbb{S}^1} \rightarrow Y$ be as $ q|_{\prod_{i\in I}{\mathbb{R}}} =p $ and identity on $ \mathbb{S}^1 $. Clearly, $ p_*\pi_1(\prod _{i \in I} {\mathbb{R}} \times \mathbb{S}^1,(\bold{0},s)) = \mathbb{Z} $ is a subgroup of the fundamental group of $ Y $ and $ q $  has unique path lifting property by the similar way of the above example. Thus $ q $ is a generalized covering homomorphism. We show that $ q $ is not a local homeomorphism, and so it is not a covering homomorphism. 
Let $ (\bold{x},s) \in {\prod_{i \in I} \mathbb{R}} \times \mathbb{S}^1  $, and let $ U $ be an open subset containing $  (\bold{x},s) $. By the definition of Tychonoff product topology, $ U $ and $ q(U) $ contain many $ \mathbb{R} $'s and $ \mathbb{S}^1 $'s except finite number, respectively. Then for some $ i \in I $ the restriction $ q|_{U_i}: \mathbb{R} \rightarrow \mathbb{S}^1 $ is not injective. This fact implies that $ q $ is not a local homeomorphism.

 \end{example}

It was shown in  \cite[Theorem 1.4.5]{Arx}  that if $G  $ is a left topological group with continuous inverse operation and $ \beta_e $ a local base of the space $ G $ at the identity element $ e $, then, for every subset $ A $ of $ G $
\[
 cl(A)=\overline{A} = \bigcap \{ AU \ | \ U \in \beta_e \} .
 \]

Recall from \cite{BroU,paper2} that for any path connected space $ X $, any subgroup $ H \leq \pi_1(X,x_0) $ and the endpoint projection map $p_H: \widetilde{X}_H \rightarrow X  $,  the fibres  $ p^{-1}_H(x) $ and $ p^{-1}_H(y) $ are homeomorphic, for every $ x,y \in X $, if and only if $ X $ is an $ SLT $ space. Note that $X$ is called a \textit{small loop transfer space ($ SLT $ space for short) } if  for every $ x,y \in X $, for every path $ \alpha : I \rightarrow X $ from $ x  $ to $ y $ and for every neighbourhood $U$ of $x$ there is a neighbourhood $V$ of $y$ such that for every loop $\beta$ in $V$ based at $y$ there is a loop $\gamma$ in $U$ based at $x$,  which is homotopic to $ \alpha\ast\beta\ast\alpha^{-1}$ relative $\dot{I}$. Clearly, every strong $ SLT $ space is also an $ SLT $ space.

For an arbitrary pointed topological sapce $ (X,x_0) $, Abdullahi et al. in \cite[Proposition 3.2 and Lemma 3.1]{paper1} showed that  the fundamental group equipped with the whisker topology, $ \pi_1^{wh}(X,x_0) $, is a left topological group and the collection $ \beta = \{i_*\pi_1(U,x_0) \ | \ U \ is \ an \ open \ subset \ of \ X \ contaning \ x_0 \} $ forms a local basis at the identity element in $ \pi_1^{wh}(X,x_0) $. Now the following proposition is obtained using these results.

\begin{proposition}
If $ (X,x_0) $ is $ SLT $ at $ x_0 $, then every closed subset  $ A  $  of $ \pi_1^{wh}(X,x_0) $ can be written as $ A = \bigcap_{U} \{Ai_*\pi_1(U,x_0) \ | \ U \ is \ an \ open \ subset \ of \ X \ contaning \ x_0 \}$.
\end{proposition}

\begin{proof}
It is easily concluded from \cite[Corollary 2.4]{paper2} that $ \pi_1^{wh}(X,x_0) $ is a topological group, since $ X $ is  $ SLT $ at $ x_0 $. Then the inverse operation is continuous. Now the result comes from \cite[Proposition 3.2 and Lemma 3.1]{paper1} and \cite[Theorem 1.4.5]{Arx} .
\end{proof}

If $ (X,x_0) $ is a strong $ SLT $ at $ x_0 $ space, then it is easily concluded from the definition that for any open neighbourhood $  U$ of $ x_0 $ in $ X $, there is an open covering $ \mathcal{U} $ of $ X $ such that $\pi (\mathcal{U}, x_0)\leq i_{*}\pi_{1}(U,x_0) $. On the other hand, by \cite[Corollary 3.10]{paper1} if $ H \leq \pi_1(X,x_0) $ is a generalized covering subgroup, then $ H $ is closed under the whisker topology on the fundamental group. The following theorem states a nice result of this fact.

\begin{theorem} \label{41}
Let $ (X,x_0) $ be a connected locally path connected space which is strong $ SLT $ at $ x_0 $, and let $ p:\widetilde{X} \rightarrow X $  be a map. The pair $ (\widetilde{X},p) $ is a generalized covering of $ X $ with $ p_*\pi_1(\widetilde{X},\tilde{x}_0)=H \leq \pi_1(X,x_0)$ if and only if $ H $ is as intersection of some covering subgroups in $ \pi_1(X,x_0) $. 
\end{theorem}

\begin{proof}
It is easily concluded from \cite[Corollary 2.13]{paper1} that the intersection of any collection of covering subgroups of the fundamental group is a generalized covering subgroup. Conversely, Since $ X $ is strong $ SLT $ at $ x_0 $, then for any open neighbourhood $  U$ of $ x_0 $ in $ X $, there is an open covering $ \mathcal{U} $ of $ X $ such that $\pi (\mathcal{U}, x_0)\leq i_{*}\pi_{1}(U,x_0) $. It means that open subgroups of the whisker topology  are covering subgroups. Thus for any open neighbourhood $ U $ of $ X $ at $ x_0 $ and any subgroup $ H $ of $ \pi_1(X,x_0) $, the subgroup $ Hi_*\pi_1(U,x_0) $ is a covering subgroup of $ \pi_1(X,x_0) $ by \cite[Theorem 2.5.13]{Span}. Finally, if $ H $ is a generalized covering subgroup, then it is closed under the whisker topology on the fundamental group (see \cite[Corollary 3.10]{paper1}) and so $ H = \overline{H} = \bigcap_{U} \{Hi_*\pi_1(U,x_0) \} $.
\end{proof}

\begin{corollary} \label{co42}
Let $ (X,x_0) $ be a connected locally path connected space which is strong $ SLT $ at $ x_0 $. A subgroup $ H \leq \pi_1(X,x_0) $ is a generalized covering subgroup if and only if $ H $ is an intersection of all covering subgroups which contain $ H $.
\end{corollary}

\begin{proof}
Assume that $ \Gamma_H  $  is the collection of all covering subgroups of $ \pi_1(X,x_0) $ containing $ H $. Clearly, $ \bigcap_{K \in \Gamma_H} K $ is a generalized covering subgroup. Conversely, by \cite[Theorem 3.1]{paper4} and \cite[Lemma 3.1]{paper1} for every $ K \in \Gamma_H $, there is an open neighbourhood $ U_K $ of $ X $ at $ x_0 $ such that $ i_*\pi_1(U_K,x_0) \leq K $. Thus $ Hi_*\pi_1(U_K,x_0) \leq K $. As mentioned above for every open neighbourhood $ U $ of $ X $ at $ x_0 $, there is an open covering $ \mathcal{U} $ of $ X $ such that $\pi (\mathcal{U}, x_0)\leq i_{*}\pi_{1}(U,x_0) \leq Hi_*\pi_1(U,x_0) $, which shows that $ Hi_*\pi_1(U,x_0) $  is a covering subgroup. Therefore, $ Hi_*\pi_1(U,x_0) \in \Gamma_H $, for every open neighbourhood $ U $ of $ X $ at $ x_0 $. It implies that $ \bigcap_{K \in \Gamma_H} K = \bigcap_{K \in \Gamma_H} Hi_*\pi_1(U_{K},x_0)=\overline{H}=H  $.
\end{proof}

%//////////////////////////////////////////////////////////////////////////////////////////////////////////////////////////////////////////////////////////////////////////////////
%\\\\\\\\\\\\\\\\\\\\\\\\\\\\\\\\\\\\\\\\\\\\\\\\\\\\\\\\\\\\\\\\\\\\\\\\\\\\\\\\\\\\\\\\\\\\\\\\\\\\\\\\\\\\\\\\\\\\\\\\\\\\\\\\\\\\\\\\\\\\\\\\\\\\\\\\\\\\\\\\\\\\\\\\\\\\\\\\\\\\

\section{Kernel of Generalized Coverings of Topological Group} \label{se4}

If $ p: \widetilde{X} \rightarrow X $ is a generalized covering map of a path connected $ SLT $ space $ X $ such that $ p_*\pi_1(\widetilde{X},\tilde{x}_0)=H \leq \pi_1(X,x_0) $, since  $\widetilde{X}  $ and $ \widetilde{X}_H $ are homeomorphic, by Lemma \ref{le33}, then for every $ x,y \in X $ the fibres $ p^{-1}(x) $ and $ p^{-1}(y) $ are homeomorphic.

\begin{corollary} \label{co37}
If $ (\widetilde{G},p) $ is a generalized covering group of a path connected topological group $ G $, then the fibres of all elements of $ G $ are homeomorphic subspaces of $ \widetilde{ G} $. 
\end{corollary}

\begin{proof}
By Theorem 2.11 in \cite{Torabi} every path connected topological group $ G $ is strong $ SLT $ and so $ SLT  $ space. Therefore, the result comes from the above assertion.
\end{proof}

\begin{proposition} \label{pr37}
Let $ (X,x_0) $ be an arbitrary topological space. If $ (\widetilde{X},p) $ is a generalized covering space, then the fibre $ p^{-1}(x_0) \subseteq \widetilde{X} $ is  totally path disconnected.
\end{proposition}

\begin{proof}
Let $ \alpha :I \rightarrow p^{-1}(x_0) $ be a path with $ \alpha(0)=\tilde{x}_0 $, and let $ C_{\tilde{x}_0} $ be the constant path in $ \tilde{x}_0 $. Clearly, $ p \circ \alpha = p \circ C_{\tilde{x}_0} = C_{x_0} $ and $ \alpha(0)=C_{\tilde{x}_0}(0) $. Since $ p $ has unique path lifting property, hence $ \alpha=C_{\tilde{x}_0} $. Therefore, $ p^{-1}(x_0)  $ has no nonconstant path; that is, it is totally path disconnected space.  
\end{proof}

\begin{corollary} \label{co38}
If $ (\widetilde{G},p) $ is a generalized covering group of connected locally path connected topological group $ G $, then $ p $ is an open epimorphism with totally path disconnected kernel.
\end{corollary}

\begin{proof}
It is easy to check that $ p $ is epic and open, since $ G $ is connected and locally path connected, respectively. The result comes from combination of Propositions \ref{pr34} and \ref{pr37}.
\end{proof}

It seems interesting whether the converse statement of the above corollary can be correct. There are some counterexamples even with extra conditions in very special case. For instance, where $ p: G \rightarrow G/H $ is an open epimorphism with totally path disconnected kernel and $ H $ is a normal subgroup of a topological group $ G $. In the next section we show that the necessary and sufficient condition to make  $ p: G \rightarrow G/H $ a generalized covering is the prodiscrete $ H $.

\begin{example} \label{ex315}
Let $ H= \mathbb{S}^1 \cap (\mathbb{Q} \times \mathbb{Q}) $ be the subgroup of Abelian topological group $ \mathbb{S}^1 $ (hence $ H \unlhd \mathbb{S}^1 $), and let $ p: \mathbb{S}^1 \rightarrow \frac{\mathbb{S}^1}{H} $ be the natural canonical map. Theorem 4.14 from \cite{Spiv} showed that $ p $ is an onto, continuous, and open map. It is easy to show that the kernel of $ p $ is equal to $ H $. Moreover, it is clear that every path in $ H $ is constant and thus it is a totally path disconnected subspace of $ \mathbb{S}^1 $. Note that since $ H $ is a dense subgroup of $ \mathbb{S}^1 $, the quotient space $ \frac{ \mathbb{S}^1 }{H} $ and its relative fundamental group, $ \pi_1(\frac{ \mathbb{S}^1 }{H},H) $, are trivial space and group, respectively. Therefore, if $  (\mathbb{S}^1,p) $ is a generalized covering group of $ \frac{ \mathbb{S}^1 }{H} $, then it is also a covering group, because the image of induced map   $ p_* \pi_1(\mathbb{S}^1,0) $ is equal to $ \pi_1(\frac{ \mathbb{S}^1 }{H},H) $. But  it is impossible, since $ H $ is not discrete.
\end{example}

 The concept of prodiscrete space  has been defined in different ways in various sources. In this paper, we define the concept of prodiscrete subgroup of a topological group as follows based on the need that we felt.
 
 \begin{definition} \label{de43}
 Let $ G $ be a topological group. We call  $ H \leq G $ a prodiscrete subgroup if there are some discrete groups $ H_i, \ i \in I $, and an isomorphism homeomorphism $ \psi : \prod_{i \in I} H_i \rightarrow H $ such that $  \psi(\prod_{i \in I, i \neq j} H_i) \trianglelefteq G $ for every $ j \in I $.
 \end{definition}
 
  Recall from \cite[Section 2.3]{BrazG} that the pull-back construction helped to show that the intersection of any collection of generalized covering subgroup is also a generalized covering subgroup. Although, there is another simple proof in \cite[Corollary 2.13]{paper1}; using pull-backs of generalized coverings will be useful to study the fibres of generalized covering maps: Let $ p: \widetilde{X} \rightarrow X $ be a generalized covering of locally path connected space $ X $, and let $ f: Y \rightarrow X $ be a map. The topological pull-back $ \diamondsuit = \{(\tilde{x},y) \in \widetilde{X} \times Y \ | \ p(\tilde{x})=f(y) \}$ is a subspace of the direct product $ \widetilde{X} \times Y $. Now for the base points $ y_0 \in Y $ and $ x_0 = f(y_0) \in X $, pick $ \tilde{x}_0 \in p^{-1}(x_0) $, and let $ f^\# \widetilde{X} $ be the path component of $ \diamondsuit $ containing $ (\tilde{x}_0,y_0) $. The projection $ f^\#p : f^\# \widetilde{X} \rightarrow Y $ with $ f^\#p(\tilde{x},y) = y $ is called the pull-back of $ \widetilde{X} $ by $ f $. Brazas showed that for a generalized covering $ p: \widetilde{X} \rightarrow X $ and a map $ f: Y \rightarrow X $, the pull-back $ f^\#p : f^\# \widetilde{X} \rightarrow Y $ is also a generalized covering \cite[Lemma 2.34]{BrazG}.
  
  \begin{theorem} \label{th43}
  If $ (G,e) $ is a connected locally path connected topological group and $ H $ is a generalized covering subgroup of $ \pi_1(G,e) $, then the kernel of the endpoint projection homomorphism $ p_H : (\widetilde{G}_H, \tilde{e}_H) \rightarrow (G,e)  $  is a prodiscrete subgroup.
  \end{theorem}
  
  \begin{proof}
   Let $  \{K_{j} \ | \  j \in J \} $  be the collection of all covering subgroups of $ \pi_1(G,e) $, which contain $ H $. By Corollary \ref{co42}, $ H= \bigcap_{j \in J} K_j $.
   For every $ j \in J  $, put $ p_{j} : (\widetilde{G}_{j},\tilde{e}_{j}) \rightarrow (G,e) $ as the relative covering homomorphism of $ K_{j} $; that is, $ K_{j}=(p_{j})_*\pi_1(\widetilde{G}_{j},\tilde{e}_{j}) $.
 Take the direct product $ \widetilde{G}=\prod_{j \in J}\widetilde{G}_{j} $ and $ p:(\widetilde{G},\tilde{e}) \rightarrow (\prod_{j  \in J}G,\bold{e}) $ is the product homomorphism defined by $ p|_{\widetilde{G}_{j}} = p_{j} $, where $ \bold{e}=(e,e,e,\dots) $. By Lemma 2.31 from \cite{BrazG}, $ p $ is a generalized covering homomorphism.
 Now consider the diagonal map $ \Delta: (G,e) \rightarrow (\prod_{j \in J}G,\bold{e}) $, and the pull-back of $ \widetilde{G} $ by $ \Delta $  is denoted by $ \Delta^{\#}p: \Delta^{\#}\widetilde{G} \rightarrow (G,e)  $ where $ \Delta^{\#}\widetilde{G} = \{(\tilde{t},g) \in \widetilde{G} \times G \ | \ p(\tilde{t}) = \Delta(g) \}$. By Lemma 2.34 of \cite{BrazG}, $ (\Delta^{\#}\widetilde{G},\Delta^{\#}p) $ is also a generalized covering group of $ G $. Let $ x_0 = ((t_j),g_0) $ be the base point of $ \Delta^{\#}\widetilde{G} $.
  At first, we show that the image of $ (\Delta^{\#}p)_* $ in $ \pi_1(G,e) $ is $ H $. This implies that $ (\widetilde{G}_H,p_H) $ is a generalized covering group by Lemma \ref{le33}, since $ (\Delta^{\#}\widetilde{G},\Delta^{\#}p) $ is a generalized covering group, and therefore $ (\Delta^{\#}\widetilde{G},\Delta^{\#}p) $ and $ (\widetilde{G}_H,p_H) $ are equivalent generalized covering groups. After that by calculating the kernel of $ \Delta^{\#}p $, we find that the kernel of  $ p_H $  is prodiscrete. 
 To do this, we claim that for a loop $ \alpha \in \Omega(G,e) $ its unique lift $ \beta \in P(\Delta^{\#}\widetilde{G},x_0) $ (such that $ \Delta^{\#}p \circ \beta  = \alpha $) is itself a loop if and only if $ [\alpha] \in H $. Let $ q: \Delta^{\#}\widetilde{G} \rightarrow \widetilde{G} $ be the projection; so that $ p \circ q = \Delta \circ \Delta^{\#}p $. By the definition it is clear that $ \beta = (\gamma,\alpha)$ for a path $ \gamma \in P(\widetilde{G},(t_j)) $. Let $ \gamma_j \in P(\widetilde{G}_j,\tilde{e}_j)  $ be the $ j $th component of $ \gamma $. Now taking the $ j $th component of the equation 
 \begin{align*}
  p \circ \gamma = p \circ q \circ \beta = \Delta \circ \Delta^{\#}p \circ \beta = \Delta \circ \alpha,
 \end{align*}
 
 which concludes that $ p_j \circ \gamma_j = \alpha $ for every $ j \in J $. Therefore 
     \begin{align*}
     \beta \ is \ a \ loop & \Leftrightarrow \ \gamma \ is \ a \ loop, & \\ &\Leftrightarrow \gamma_j \ is \ a \ loop \ for \ every \ j  \in J, & \\ & \Leftrightarrow [\alpha] \in (p_j)_*(\pi_1(\widetilde{G}_j,\tilde{e}_j) = K_j \ for \ every \ j \in J, & \\ & \Leftrightarrow [\alpha] \in H.
\end{align*}

Clearly, $ ker( \Delta^{\#}p)  = \{(\tilde{t},g) \ | \  \Delta^{\#}p(\tilde{t},g) = e \}$. This support that $ g=e $ and $ p(\tilde{t})= \Delta(e) = \bold{e}=(e,e,\dots)$. Hence $ p_j((\tilde{t})_j) = e $ for every $ j \in J $. It implies that   $ \tilde{t} \in ker(p) = \prod_{j \in J} ker(p_j)$. Then $ ker( \Delta^{\#}p)  = \{(\tilde{t},e) \ | \ \tilde{t} \in \prod_{j \in J} ker(p_j)  \} $. Now since for every $ j \in J  $, $ p_j : \widetilde{G}_j \rightarrow G $ is a covering homomorphism; thus its kernel is discrete. Let $ h: \Delta^{\#}\widetilde{G} \rightarrow \widetilde{G}_H $ be the homeomorphism obtained from Lemma \ref{le33} (see Diagram \ref{fi1}). It is easy to check that $ h $ is an isomorphism homeomorphism. Therefore $ ker(p_H) $ is a prodiscrete subgroup of $ \widetilde{G}_H $, since 
$ker (\Delta^{\#}p)$ is prodiscrete.

\begin{figure}[h!]
  \centering
   \[
\xymatrix{
{}\save[]+<2cm,-1.2cm>*{%
\circlearrowleft}
\restore
{}\save[]+<4.5cm,-1cm>*{%
\circlearrowright}
\restore
& (\Delta^{\#}\widetilde{G}, x_0)\ar[r]^{\!\!\!\!\!\!\!\!\!\!\!\!\!q}\ar[d]^{\Delta^{\#}p} \ar@{-->}[ld]_{h}& (\widetilde{G}=\mathop{\prod}\limits_{j\in J} \widetilde{G}_j, e) \ar[d] ^{p} \\
(\widetilde{G}_H , e_H )  \ar[r]_{p_H} & (G, e) \ar[r]_{\!\!\!\!\!\!\!\!\!\!\!\!\!\!\!\!\!\!\!\!\!\!\!\!\Delta} & \Big(\mathop{\prod}\limits_{j\in J} G, (e,e,\dots)\Big)
}
\]
 \caption{Diagram of $ \widetilde{G}_H $.}\label{fi1}
        \end{figure}
         \end{proof}
  
  \begin{theorem} \label{th44}
 If $ G $ is a topological group and $ H $ is a prodiscrete normal subgroup of $ G $ such that $ \frac{G}{H} $ is a connected locally path connected space, then the natural canonical homomorphism $ \varphi : G \rightarrow \frac{G}{H} $ is a generalized covering homomorphism.
  \end{theorem}
  
 \begin{proof}

  It is clear that $ \frac{G}{H} $ is a topological group and by Theorem 4.14 from \cite{Spiv} $ \varphi $ is an open epimorphism. Let $ \psi : \prod_{j \in J}H_j \rightarrow H $ be the isomorphism homeomorphism of Definition \ref{de43}. For every  $ j \in J $  put $ K_j = \psi(\prod_{i \in J, i \neq j} H_i) $.  Let $ p_j : \frac{G}{K_j} \rightarrow \frac{G}{H}$ be the natural canonical homomorphism. It is clear from the following diagram that $ p_j $ is continuous and open, because $ \varphi $  and $ \varphi ' $ are continuous and open. 
  
  %\begin{displaymath} \label{fi2}
  \begin{figure}[h!]
  \centering
  \[  \xymatrix{
  & G\ar[ld]_{\varphi^{'}}  \ar[rd] ^{\varphi} & \\
        \frac{G}{K_j} \ar[rr]_{p_j}      & & \frac{G}{H}}
        \]
        \label{fi2}
        \end{figure}
%\end{displaymath}

  Moreover, for every $ j \in J $, $ p_j  $ is also an epimorphism with the kernel  $ \frac{H}{K_j}$. Since $ \psi( H_j) $ is discrete, then by Remark \ref{reNew} the pair $ (\frac{G}{K_j},p_j) $ is a covering group of $ \frac{G}{H} $ for every $ j \in J $. 
  
  Take the direct product $ \widetilde{G}=\prod_{j \in J} \frac{G}{K_j} $ and $ p:(\widetilde{G},\tilde{e}) \rightarrow (\prod_{j \in J}\frac{G}{H},\bold{e}) $ is the product homomorphism; that is, for every $ j \in J $, $ p|_{\frac{G}{K_j}} = p_{j} $ and $ \bold{e}=(H,H,\dots) $. It   implies from Lemma 2.31 of \cite{BrazG}  that $ p $ is a generalized covering homomorphism.
 Now consider the diagonal map $ \Delta: (\frac{G}{H},H) \rightarrow (\prod_{j \in J}\frac{G}{H},\bold{e}) $, and the pull-back of $ \widetilde{G} $ by $ \Delta $ is denoted by $ \Delta^{\#}p: \Delta^{\#}\widetilde{G} \rightarrow (\frac{G}{H},H)  $ where $ \Delta^{\#}\widetilde{G} = \{((g_jK_j)_{j \in J},gH) \in \widetilde{G} \times \frac{G}{H} \ | \ p((g_jK_j)_{j \in J}) = \Delta(gH) \}$. By Lemma 2.34 of \cite{BrazG}, $ (\Delta^{\#}\widetilde{G},\Delta^{\#}p) $ is also a generalized covering group of $ \frac{G}{H} $. 
 To complete  the proof, it is enough to show that $ (\Delta^{\#}\widetilde{G},\Delta^{\#}p) $, and $ (G,e) $ are equivalent generalized covering groups of $ \frac{G}{H} $. 
Let $ x_0 = ((K_j)_{j \in J},H) $ be the base point of $ \Delta^{\#}\widetilde{G} $, and define homomorphism $ \theta: (G,e) \rightarrow  (\Delta^{\#}\widetilde{G},x_0) $ with $ \theta(g)= ((gK_j)_{j \in J},gH) $. Since for every $ g \in G $, $ \Delta^{\#}p \circ \theta(g) =\Delta^{\#}p((gK_j)_{j \in J},gH)= gH $, hence the right triangle of Diagram \ref{fi3} is commutative.   

 \begin{figure}[h!]
  
    \[
\xymatrix{
{}\save[]+<.5cm,-.65cm>*{%
\circlearrowleft}
\restore
{}\save[]+<1.5cm,-1.2cm>*{%
\circlearrowright}
\restore
{}\save[]+<3.4cm,-1cm>*{%
\circlearrowleft}
\restore
\widetilde{G} \times \frac{G}{H}  \ar[r]^{\upsilon}  & \Delta^{\#} \widetilde{G} \ar[r]^{\!\!\!\!\!\!\!\!\!\!\!\!\!q}\ar[d]^{\Delta^{\#}p} \ar@{<--}[ld]_{\theta}& \widetilde{G}=\mathop{\prod}\limits_{j\in J} \frac{G}{K_j} \ar[d] ^{p} \\
G  \ar[r]_{\varphi} \ar[u]^{\sigma} & \frac{G}{H} \ar[r]_{\!\!\!\!\!\!\!\!\!\!\Delta} & \mathop{\prod}\limits_{j\in J} \frac{G}{H}
}
\]
       \caption{Commutative diagram} \label{fi3}
        \end{figure}

We show that $ \theta $ is an isomorphism homeomorphism. Clearly,
\[
 ker(\theta)=\{ g \in G \ |  \ ((gK_j)_{j \in J},gH) = x_0 = ((K_j)_{j \in J},H) \} = \bigcap_{j \in J}K_j.
 \]
 Since $ \bigcap_{j \in J}K_j $ is the trivial subgroup, thus $ \theta $ is one to one. To show that $ \theta $ is onto, let $ ((g_jK_j)_{j \in J},gH) $ be an arbitrary element of $\Delta^{\#}\widetilde{G}  $, and let $ \pi_j : H_j \rightarrow \prod_{j \in J} H_j $ be the inclusion map. By the definition of $\Delta^{\#}\widetilde{G}  $, for every $ j \in J $, there exists $ h_j \in  H_j $ such that $ g_j = g \psi(\pi_j(h_j)) $. If $ h = (h_j)_{j \in J} \in \prod_{j \in J} H_j $ and $ s \in J $, then $ h= \pi_s(h_s).(l_j)_{j \in J} $ where $ l_j = h_j $ if $ j \neq s $ and $ l_s=1 $. Hence, 
 \[ 
 \psi(h)= \psi(\pi_s(h_s)(l_j)_{j \in J}) = \psi(\pi_s(h_s)) \psi((l_j)_{j \in J}).
 \]
  Then $ \psi(h)K_s = \psi(\pi_s(h_s))K_s $. Therefore,
  \[
   \theta(g \psi(h)) = ((g\psi(h)K_j)_{j \in J},ghH) = ((g\psi(\pi_j(h_j))K_j)_{j \in J},gH) = ((g_jK_j)_{j \in J},gH).
   \]
 
 It is clear from Diagram \ref{fi3} that  $ \theta = \upsilon \circ \sigma $, where $ \sigma : G \rightarrow \widetilde{G} \times \frac{G}{H} $ is the product of natural canonical maps and  $ \upsilon : \widetilde{G} \times \frac{G}{H}  \rightarrow \Delta^{\#}\widetilde{G} $  is the quotient map. Hence  $ \theta $ is continuous because $ \sigma $ and $ \upsilon $ are continuous.
The continuity of $ \theta ^{-1} $ follows from openness of $ \varphi $ and continuity of $ \Delta^{\#}p $.
 \end{proof}
  
  \begin{corollary} \label{co45}
 Let $ G $ be a topological group, and let $ H \trianglelefteq G$ be such that the quotient group $ \frac{G}{H} $ is connected locally path connected. The natural canonical homomorphism $ \varphi : G \rightarrow \frac{G}{H} $ is a generalized covering homomorphism if and only if $ H $ is a prodiscrete subgroup of $ G $.
  \end{corollary}
  
  \begin{proof}
  It is clear that  $ \frac{G}{H} $ is a topological group. If $ \varphi $ is a generalized covering homomorphism, then by Theorem \ref{th43} the kernel of $ \varphi  $, $ H $, is prodiscrete. The converse statement is obtained from Theorem \ref{th44}.
%Since $ H $ is a normal subgroup of topological group $ G $, then the quotient group $ \frac{G}{H} $ is also a topological group. By Theorem 4.14 from \cite{Spiv}, $ \varphi $ is an open epimorphism. One can easily show that the Kernel of $ \varphi  $ is the subgroup $ H $ and conclude the result from Corollary \ref{co211}.
\end{proof}
  
  \begin{corollary} \label{th46}
  Let  $ G $ be connected locally path connected,  and let $  p: \widetilde{G} \rightarrow G $ be a homomorphism on topological groups. The pair $ (\widetilde{G},p) $ is a generalized covering group of $ G $ if and only if $ p $ is an open epimorphism and the kernel of $ p $ is a prodiscrete subgroup of $ \widetilde{G}  $.
  \end{corollary}
  
  \begin{proof}
  Let $  p: \widetilde{G} \rightarrow G $ be a generalized covering homomorphism, and let $ H=p_* \pi_1(\widetilde{G},\tilde{e}) $. By Lemma \ref{le33}, there is a homeomorphism $ h: \widetilde{G} \rightarrow  \widetilde{G}_H$ such that $ p_H \circ h = p $. Since $ p_H $ is open onto, then also $ p $ is. Moreover, Theorem \ref{th43} implies that the kernel of $ p_H $ and  so $ p $ is prodiscrete.
  Conversely, consider   $ \varphi : \widetilde{G}  \rightarrow \frac{\widetilde{G}}{ker(p)}$   as the natural canonical homomorphism. Since $ \varphi $ is onto, then the homomorphism $ \theta: \frac{\widetilde{G}}{ker(p)} \rightarrow G $ is an isomorphism homeomorphism, where $ \frac{\widetilde{G}}{ker(p)} $ is equipped with the quotient topology. Since $ ker(p) $ is a prodiscrete subgroup of $ \widetilde{G}  $ and $ \frac{\widetilde{G}}{ker(p)} $ is connected locally path connected (derived from $ G $ is connected locally path connected), it implies from Corollary \ref{co45} that  $ (\widetilde{G},\varphi)  $ is a generalized covering group of  $ \frac{\widetilde{G}}{ker(p)} $. Therefore,  $ (\widetilde{G},p=\varphi \circ \theta) $ is a generalized covering group of $ G $. 
  \end{proof}

By Remark \ref{re} and Corollary \ref{th46} one can easily conclude  the following corollary which was promised in Remark \ref{reNew}.
  
    \begin{corollary} \label{co310}
  Let  $ G $ be a connected locally path connected topological group, and let $  p: \widetilde{G} \rightarrow G $ be a continuous map. The pair $ (\widetilde{G},p) $ is a generalized covering group of $ G $ if and only if $ \widetilde{G} $ is a topological group and $ p $ is an open epimorphism with prodiscrete kernel.
  \end{corollary}
  
  \begin{example}
  It is easy to show that the kernel of generalized covering maps $ p:\prod_{i \in I} \mathbb{R} \rightarrow Y  $ and $ q: \prod_{i \in I}{{\mathbb{R}} \times \mathbb{S}^1} \rightarrow Y$ in Examples \ref{ex3.12} and \ref{ex313}, respectively, both are $  \prod_{i \in I} \mathbb{Z} $ which is the product of discrete groups. For another example, let $ r_n :\mathbb{S}^1 \rightarrow \mathbb{S}^1  $ with $ r_n (z)=z^n $ be the well-known covering homomorphism. As mentioned before, the product homomorphism $ r = \prod_{n \in \mathbb{N}} r_n: \prod_{n \in \mathbb{N}} \mathbb{S}^1 \rightarrow \prod_{n \in \mathbb{N}} \mathbb{S}^1 $ is a generalized covering homomorphism and $ ker(r) =  \prod_{n \in \mathbb{N}} A_n $ where $ A_n = \{ z \in \mathbb{S}^1 \ | \ z^n=1 \} $, which is also a product of discrete groups. Moreover, the kernel of $ p $ in Example \ref{ex315} is not prodiscrete.
  \end{example}

   In \cite{Biss}, Biss  investigated on a kind of fibrations which is called \textit{rigid covering fibration} with properties similar to covering spaces. Nasri et al. \cite{Nasri} simplified the definition of rigid covering fibration such as a fibration $ p: E \rightarrow X $ with unique path lifting (unique lifting with respect to paths) property and concluded from \cite[Theorem 2.2.5]{Span} that a fibration $ p: E \rightarrow X $ is a rigid covering fibration if and only if each fibre of $ p $ is totally path disconnected. On the other hand it was shown in \cite[Theorem 2.4.5]{Span} that a rigid covering fibration has the unique lifting property and so it is a generalized covering spaces. Although, the converse statement may not hold, in general, we show that it is right in the case of connected locally path connected topological groups.
   
   \begin{proposition} \label{pr312}
   For  a connected locally path connected topological group $ G $, $ (\widetilde{G},p) $ is a generalized covering group  if and only if $ p:  \widetilde{G} \rightarrow G$ is rigid covering fibration. 
\end{proposition}    
  
  \begin{proof}
Let $ p:  \widetilde{G} \rightarrow G $ be a generalized covering homomorphism, and let $ H= p_* \pi_1(\widetilde{G},\tilde{e}) $. By Lemma \ref{le33}, $ (\widetilde{G},p) $ and $ (\widetilde{G}_H,p_H) $ are equivalent generalized covering groups of $ G $, and so the kernel of $ p_H $ is a prodiscrete subgroup, since the kernel of  $ p $ is prodiscrete. On the other hand, it was shown in \cite[Section 3]{paper1} that the kernel of $ p_H $ and the left coset space $ \frac{\pi_1^{wh}(G,e)}{H} $ are homeomorphic, which implies that $ \frac{\pi_1^{wh}(G,e)}{H} $ is a prodiscrete space. Then $ \frac{\pi_1^{wh}(G,e)}{H} $ is totally path disconnected; that is, $ \frac{\pi_1^{wh}(G,e)}{H} $ has no nonconstant paths. As mentioned above, $ \frac{\pi_1^{qtop}(G,e)}{H} =  \frac{\pi_1^{wh}(G,e)}{H}  $ has no nonconstant paths. Now use  \cite[Theorem 4.3]{Biss} which guarantees the existence of a rigid covering fibration $ q: E \rightarrow G $ with $ p_*\pi_1(E)=H $. Since every rigid covering fibration has unique lifting property (see \cite[Theorem 2.4.5]{Span}), there is a homeomorphism $ g $ between $ E $ and $ \widetilde{G}_H $ (Lemma \ref{le33}). Therefore, it is done as shown in Diagram \ref{fi4}.

\begin{figure}[h!] \label{fi4}
\centering
  \[
\xymatrix{
{}\save[]+<1.2cm,-.5cm>*{%
\circlearrowright}
\restore
{}\save[]+<2.4cm,-.5cm>*{%
\circlearrowleft}
\restore
\widetilde{G} \ar[rd]_{p} \ar[r]^h  & \widetilde{G}_H  \ar[d]^{p_H} & E  \ar[ld]^{q} \ar[l]_{g}\\
 & (G, e)  & 
}
\]
\caption{Diagram}
\end{figure}
The converse statement obtains from \cite[Theorem 2.4.5]{Span}.
  \end{proof}
  
  The following corollary is the immediate result of the definition of rigid covering fibrations and Proposition \ref{pr312}.
 \begin{corollary}
   If $ G $ is a connected locally path connected topological group and $ (\widetilde{G},p) $ is a generalized covering group of $ G $, then $ p:  \widetilde{G} \rightarrow G$ is a fibration.
  \end{corollary}

\
\

\
\\
\\

\end{document}